\documentclass{article}

\usepackage{amsfonts,amssymb}
\usepackage{amsmath}
\usepackage{latexsym}
\usepackage{times}
\usepackage{cite}
\usepackage{amsthm}

\usepackage{times}

\newtheorem{thm}{Theorem}
\newtheorem{thl}[thm]{Lemma}
\newtheorem{thp}[thm]{Proposition}
\newtheorem{thc}[thm]{Corollary}

\theoremstyle{definition}

\newcommand{\ang}{\triangleright}

\newcommand{\N}{\mathbb{N}}
\newcommand{\R}{\mathbb{R}}
\newcommand{\C}{\mathbb{C}}
\newcommand{\hsp}{{\hspace{-1pt}}}
\newcommand{\hs}{{\hspace{1pt}}}

\newcommand{\Hh}{{\mathcal{H}}}
\newcommand{\cO}{{\mathcal{O}}}

\newcommand{\K}{{\mathcal{K}}}
\newcommand{\U}{{\mathcal{U}}}   \newcommand{\cU}{{\mathcal{U}}}

\newcommand{\X}{{\mathcal{X}}}
\newcommand{\A}{{\mathcal{A}}}

\newcommand{\F}{{\mathcal{F}}}

\newcommand{\cL}{\mathcal{L}}

\newcommand{\fA}{\mathfrak{A}}

\newcommand{\FF}{\mathbb{F}}
\newcommand{\FFD}{\mathbb{F}(D)}

\newcommand{\BB}{\mathbb{B}}
\newcommand{\BBA}{\mathbb{B}_1(\fA)}

\newcommand{\im}{\mathrm{i}}
\newcommand{\Lin}{{\mathrm{Lin}}}

\newcommand{\vare}{\varepsilon}

\newcommand{\ov}{\overline}
\newcommand{\rf}[1][]{\textup{\eqref{#1}}}

\newcommand{\sut}{{\cU}_q(\mathrm{su}_{1,1})}       

\newcommand{\slr}{{\cU}_q(\mathrm{sl}_2(\R))}
\newcommand{\slrn}{{\cU}_q(\mathrm{sl}_{n+1}(\R))}
\newcommand{\slrm}{{\cU}_q(\mathrm{sl}_{n}(\R))}
\newcommand{\sltn}{{\cU}_q(\mathrm{sl}_{n+1})}

\newcommand{\tr}{\mathrm{tr}}

\newcommand{\qh}{\A_q(1;\R)}
\newcommand{\qhn}{\A_q(n;\R)}
\newcommand{\ld}{\mathcal{L}^+(D)}            

\newcommand{\dd}{\mathrm{d}}

\title{An operator-theoretic approach to invariant
 integrals on quantum homogeneous \boldmath{${\rm sl}_{n+1}(\R)$}-spaces
}
\author{Osvaldo {\sc Osuna Castro} and Elmar {\sc Wagner}}

\date{ }

\makeatletter

\begin{document}

\maketitle

\thispagestyle{empty}

\mbox{ }\\[-24pt]
\centerline{{\small Instituto de F\'isica y Matem\'aticas}}
\centerline{{\small Universidad Michoacana de San Nicol\'as de Hidalgo,
Morelia, M\'exico }}
\centerline{{\small E-mail: osvaldo@ifm.umich.mx, elmar@ifm.umich.mx}}

\begin{abstract}
We present other examples illustrating the operator-theoretic approach to invariant
integrals on quantum homogeneous spaces developed by K\"ursten
and the second author. The quantum spaces are chosen such that their coordinate
algebras do not admit bounded Hilbert space representations and their
self-adjoint generators have continuous spectrum.
Operator algebras of trace class operators are associated to the coordinate
algebras which allow interpretations as rapidly decreasing functions and as
finite functions. The invariant integral is defined as a trace functional
which generalizes the well-known quantum trace.
We argue that previous algebraic methods would fail for these examples.   
\end{abstract}
%
{\small
}

%
\section{Introduction}
%
                                                          \label{sec-mot}

In a series of papers, Shklyarov, Sinel'shchikov and Vaksman studied
non-commutative analogues of bounded symmetric domains of
non-compact type \cite{VK1,VKcov,VK2,VKPrag,VK3}.
The corresponding non-commutative algebras
can be described by stating the commutation relations of
their generators which are viewed as coordinate functions.
As a first step toward the development of function theory and harmonic analysis,
the authors defined covariant differential calculi and invariant integrals
on these algebras. Naturally, because of the non-compactness, the
invariant integral does not exist on polynomial functions in the
coordinates. To circumvent the problem, the authors introduced
algebras of finite functions on the quantum space.
Basically, this was done by first considering a faithful Fock-type representation,
where some distinguished self-adjoint operators have a discrete joint spectrum,
and then adjoining functions with finite support (on the joint spectrum)
to the (represented) algebra of coordinate functions \cite{KW,VK1,VK2,VKPrag}.
The algebra of finite functions can be equipped with a
symmetry action of a quantum group, and the invariant integral is defined as
a generalization of the (well-known) quantum trace.

As is customary when defining quantum spaces,
the approach of Shklyarov, Sinel'\-sh\-chikov and Vaksman
is almost completely algebraic, using only a minimal amount of topology.
On the other hand, Hilbert space representations of the coordinate algebras
provide a systematic tool for exploring topological questions.
This observation was the starting point for the operator-theoretic approach
to invariant integrals proposed by K\"ursten and the second author \cite{KW}.
The use of operator-theoretic methods has several advantages.
For instance, it allows to adjoin a wider class of integrable functions to the
coordinate algebra, one can prove density and continuity results, and
it works for all Hilbert space representations in the same way.
In particular, the operator-theoretic methods apply to algebras
which to not admit bounded representations and where the
``distinguished self-adjoint operators'' do not have a discrete spectrum.

The objective of the present paper is to provide an example
that can easily be treated by
the operator-theoretic approach developed in \cite{KW} but
for which purely algebraic methods seem to fail.
In order to keep the exposure as close as possible to the
previous paper \cite{KW}, we take the same coordinate algebras
and change only the involution and the values of the deformation
parameter. The corresponding quantum spaces are known as
real quantum hyperboloid and real $q$-Weyl algebra.
The difference of
the involution has profound consequences:
First, there do not exist bounded Hilbert space *-rep\-re\-sen\-ta\-tions, and second,
the *-rep\-re\-sen\-ta\-tions of the coordinate algebra are determined by self-adjoint
operators with continuous spectrum \cite{S3,S4}.
This impedes the description of integrable functions as finite
functions of the self-adjoint generators since these operators are not of trace class
and the generalized quantum trace will not exist on them.
Nevertheless, we shall see that the methods from \cite{KW} can be applied 
even in this situation. 

Let us briefly outline the main ideas of \cite{KW}.
Suppose we are given a Hopf *-algebra $\U$ acting on a
*-algebra of coordinate functions $\X$. It is natural to require that the
action respects the Hopf *-structure of $\U$ and the multiplicative structure
of $\X$. In other words, we assume that $\X$ is a $\U$-module *-algebra.
Let $\pi:\X\rightarrow \ld$ be a *-representation
of $\X$ into a *-algebra
of closeable operators on a pre-Hilbert space $D$.
Starting point of the operator-theoretic approach is
an \emph{operator expansion} of the action. This means that for each
$Z\in\U$ there exists a
finite number of operators $L_i,R_i\in \ld$ such that
\begin{equation}                                     \label{oprep}
     \pi(Z\ang x)=\sum_i L_i\pi(x)R_i,\quad \ x\in\X,
\end{equation}
where $\ang$ denotes the (left) $\U$-action on $\X$.
Obviously, it suffices to know the operators $L_i$, $R_i$ for a set of
generators of $\U$. The operator expansion allows us to extend
the action to the *-algebra $\ld$ turning it into a
$\U$-module *-algebra. Inside $\ld$, we find two $\U$-module *-subalgebra
of particular interest. The first one is the algebra of finite rank operators
and will be considered as the algebra of finite functions associated to $\X$.
The second one is an algebra of trace class operators which is stable under
multiplication by operators from the operator expansion. This algebra will
be viewed as an algebra of functions which vanish
sufficiently rapidly at ``infinity''. On both algebras,
the invariant integral can be defined by a trace formula
resembling the quantum trace of finite dimensional representations of $\U$.

To make the paper more readable, we first discuss in Section \ref{sec-qh}
the lowest dimensional case of a $q$-Weyl algebra, namely the so-called
real quantum hyperboloid. The general case will be treated
in Section \ref{sec-qhn}.

%
\section{Preliminaries}
                                                          \label{algprel}
Throughout the paper,
$q$ is a complex
number such that $|q|=1$ and $q^4\neq1$.
The letter $\im$ stands for the imaginary unit, and
we set $\lambda := q-q^{-1}$. Note that $\lambda\in\im \R$.

Let $\cU$ be a Hopf *-algebra with
comultiplication $\Delta$, counit $\varepsilon$, and antipode $S$.
Adopting Sweedler's notation, we write  $\Delta(x)=x_{(1)}\otimes x_{(2)}$
for $x\in\U$.  An *-algebra $\X$ is called {\it left $\cU$-module *-algebra} \cite{KS}
if there is a left $\cU$-action $\ang$ on $\X$ such that
\begin{equation}                                     \label{modalg}
f\ang (xy)=(f_{(1)}\ang x)(f_{(2)}\ang y),\quad (f\ang x)^* =S(f)^*\ang x^*,\quad
x,y\in\X,\ f\in\cU.
\end{equation}
For unital algebras, one additionally requires
\begin{equation}                                      \label{modeins}
   f\ang 1=\varepsilon(f)1,\quad f\in\cU.
\end{equation}
By an {\it invariant integral} we mean a linear functional $h$ on $\X$
satisfying
\begin{equation}                                       \label{invint}
   h(f\ang x)=\varepsilon(f) h(x),\quad x\in\X,\ f\in \cU.
\end{equation}
Synonymously, we refer to it as {\it $\cU$-invariant}.

In this paper, the Hopf *-algebra under consideration will be a
$q$-deformation of the universal enveloping algebra of ${\rm sl}_{n+1}(\C)$
with non-compact real form. Let $n\in\N$.
Recall that the Cartan matrix $(a_{ij})_{i,j=1}^n$ of
${\rm sl}_{n+1}(\C)$ is given by
$a_{j,j+1}=a_{j+1,j}=-1$ for $j=1,\ldots,n\hsp - \hsp 1$,\,
$a_{jj}=2$ for $j=1,\ldots,n$,
and $a_{ij}=0$ otherwise.
The Hopf *-algebra $\slrn$ is
generated by $K_j$, $K_j^{-1}$, $E_j$, $F_j$,\,
$j\hsp =\hsp 1,\ldots, n$, with relations  \cite{KS}
\begin{align*}                                
 &K_iK_j=K_jK_i,\  K_j^{-1}K_j=K_jK_j^{-1}=1,\
 K_iE_j=q^{a_{ij}}E_jK_i,\   K_iF_j=q^{-a_{ij}}F_jK_i,\\
& E_iE_j-E_jE_i=0,\ \, i \neq  j \pm 1,
\ \ 
E_j^2E_{j\pm1}-(q +  q^{-1})E_jE_{j\pm1}E_j+E_{j\pm1}E_j^2=0,\\%
&F_iF_j-F_jF_i=0,\ \,i \neq  j \pm 1,
\quad
F_j^2F_{j\pm1}-(q +  q^{-1})F_jF_{j\pm1}F_j+F_{j\pm1}F_j^2=0,\\
 & E_iF_j-E_jF_i=0,\ \, i \neq  j,  \quad
     E_jF_j-F_jE_j=\lambda^{-1}(K_j-K_j^{-1}),\quad j =  1,\ldots, n, 
\end{align*}
comultiplication, counit and antipode given by 
\begin{eqnarray*}            \label{delta}
&\hsp\hsp\Delta (E_j)\hsp =\hsp E_j\hsp\otimes\hsp 1\hsp +\hsp K_j\hsp\otimes\hsp E_j, \quad\!\!
\Delta(F_j)\hsp =\hsp F_j\hsp\otimes\hsp K_j^{-1}\hsp+\hsp1\hsp\otimes\hsp F_j,\quad\!\!
\Delta(K_j)\hsp =\hsp K_j\hsp\otimes\hsp K_j,&\\
&\varepsilon(K_j)=\varepsilon(K_j^{-1})=1,\quad
\varepsilon(E_j)=\varepsilon(F_j)=0,&\\
 &S(K_j)=K_j^{-1},\quad  S(E_j)=-K_j^{-1}E_j,\quad S(F_j)=-F_jK_j,&
\end{eqnarray*}
and involution
\begin{equation*} 
K_i^*=K_i,\quad E_j^*=E_j,\quad F_j^*=F_j.
\end{equation*}

If $n=1$, we write $K$, $K^{-1}$, $E$, $F$ rather than
$K_1$, $K_1^{-1}$, $E_1$, $F_1$. These generators
are hermitian and satisfy the relations
\begin{eqnarray*}                                    
&KK^{-1}=K^{-1}K=1,\quad KEK^{-1}=q^2E,\quad KFK^{-1}=q^{-2}F,&\\
        &EF-FE=(K-K^{-1})/(q-q^{-1}).&
\end{eqnarray*}

We turn now to operator-theoretic preliminaries.
Let $\Hh$ be a Hilbert space.
For a closable densely defined operator $T$ on $\Hh$,
we denote by $D(T)$, $\bar{T}$, $T^*$ and $|T|$ its domain,
closure, adjoint and modulus, respectively.
Given a dense linear subspace $D$ of $\Hh$, we set
$$
   \ld:=\{\,x\in {\rm End}(D)\,;\,D\subset D(x^*),\ x^*D\subset D\,\}.
$$
Clearly, $\ld$ is a unital algebra of closeable operators.
It becomes a *-algebra if we define the involution by
$x\mapsto x^+:=x^*\lceil D$.
Since it should cause no confusion, we shall continue to write
$x^*$ in place of $x^+$. Unital *-subalgebras of $\ld$ are
called {\it O*-algebras}.

Given an O*-algebra $\fA$, set
\begin{equation}                                  \label{BA}
 \BB_1(\fA):=\{\,t\in \ld\,;\, \bar t\Hh\subset D,\ \bar t^*\Hh \subset D,\
  \ov{atb}\ \mbox{is\ of\ trace\ class\ for\ all}\ a,b\in\fA\,\}.
\end{equation}
It follows from \cite[Lemma 5.1.4]{S} that $\BB_1(\fA)$ is a *-subalgebra
of $\ld$. Next, let
\begin{equation}                                  \label{FD}
  \FF(D):=
\{\, x\in \ld\,;\, \bar{x}\ \mbox{is\ bounded},\
 {\rm dim}(\bar x \Hh)<\infty,\ \bar{x}\Hh\subset D,\ \bar{x}^*\Hh\subset D\,\}.
\end{equation}
Note that each element $A\in\FF(D)$ can be written as $A=\sum^{n}_{i=1}\alpha_i e_i\otimes f_i$,
where $n\in\N$, $\alpha_i\in\C$, $f_i,e_i\in D$,
and $(e_i\otimes f_i)(x):=f_i(x)e_i$ for $x\in D$.
Obviously, $\FF(D)\subset\BB_1(\fA)$ and
$1\notin \BB_1(\fA)$ if dim$(\Hh)=\infty$.

By a *-rep\-re\-sen\-ta\-tion $\pi$ of a *-algebra $\X$ on a domain $D$,
we mean a *-ho\-mo\-mor\-phism $\pi : \fA \rightarrow \ld$.
For notational simplicity, we usually suppress  the rep\-re\-sen\-ta\-tion
and write $x$ instead of $\pi(x)$ when no confusion can arise.
%

%
\section{Real quantum hyperboloid}
                                                            \label{sec-qh}
The *-algebra $\qh$ of coordinate functions on the
real quantum hyperboloid is generated by two hermitian
elements $x$ and $y$ fulfilling
\begin{equation}                                   \label{xy}
  xy-q^2yx=1-q^2.
\end{equation}
Set
\begin{equation}
   Q:=\lambda^{-1}(yx-xy)=q(1-yx).
\end{equation}
Since $\lambda\in\im \R$, we have $Q^{*}=Q$.
The commutation relations of $Q$ with $x$ and $y$ are given by
\begin{equation}                            \label{Qxy}
  Qy=q^2yQ,\quad Qx=q^{-2}xQ.
\end{equation}

Note that the two generators of the quantum disc \cite{KL} satisfy the
same relation as $x$ and $y$, only the involution is
different ($x^{*}=y$).
In \cite[Section 8]{VK3}, one can find an explicit construction of a
$\sut$-action on the quantum disc.
Replacing in \cite{VK3} the involution on $v_{\pm}(0)$ by
$v_{\pm}(0)^*=v_{\pm}(0)$ and performing the construction for
$\slr$ yields the following $\slr$-action on $\qh$:
\begin{equation}                                     \label{qhy}
 K^{\pm1}\ang y=q^{\pm 2}y,\quad E\ang y=\im qy^2,\quad
 F\ang y=\im,
\end{equation}
\begin{equation}                                     \label{qhx}
 K^{\pm}\ang x=q^{\mp2}x,\quad E\ang x=-\im q^{-1},\quad
 F\ang x=-\im q^{2}x^{2}.
\end{equation}
By \cite{VK3}, this action turns $\qh$ into a $\slr$-module
*-algebra.

The crucial step toward
an invariant integration theory on the
quantum hyperboloid is an operator expansion \eqref{oprep} of the action.
This will be done in the next lemma.
\begin{thl}                           \label{qh-act}
Assume that $\pi : \qh\rightarrow \ld$ is a *-rep\-re\-sen\-ta\-tion
of $\qh$ such that $Q^{-1}\in\ld$.
Set
\begin{equation}                              \label{AB}
A:=-\im \lambda^{-1}y,\quad\, B:= -\im\lambda^{-1} q^{-1} Q^{-1}x.
\end{equation}
The formulas
\begin{align}
 & K\ang f =QfQ^{-1},\quad K^{-1}\ang f=Q^{-1}fQ,\label{qhact1}\\
 & E\ang f=Af-QfQ^{-1}A,                          \label{qhact2}\\
 & F\ang f= BfQ-q^2fQB                           \label{qhact3}
\end{align}
applied to $f\in\qh$ define an operator expansion of the action $\ang$
on $\qh$.
If $f$ is taken from $\ld$, then the same formulas turn the
$O^*$-algebra $\ld$ into a $\slr$-module *-algebra.
\end{thl}
\begin{proof}
First we show that \rf[qhact1]--\rf[qhact3] define an
action on $\ld$ which turns $\ld$ into a $\slr$-module *-algebra.
A straightforward calculation shows that
\begin{equation}                                \label{ABQ}
  QA=q^2AQ,\quad QB=q^{-2}BQ,\quad AB-BA=-\lambda^{-1}Q^{-1}.
\end{equation}
That the action is well-defined can be proved by direct verification
using \rf[ABQ]. As a sample,
\begin{align*}
(EF-FE)\ang f &= ABfQ+QfBA-BAfQ-QfAB\\
            &=(AB-BA)fQ-Qf(AB-BA) \\
                  &= \lambda^{-1} (QfQ^{-1}-Q^{-1}fQ)
             =\lambda^{-1}(K-K^{-1})\ang f.
\end{align*}
for all $f\in\ld$.

To show that $\ld$ is a $\slr$-module *-algebra, it suffices to verify
\rf[modalg] and \rf[modeins] for the generators $E$, $F$, $K$ and $K^{-1}$.
Observe that
$K^{\pm1} \ang 1=1=\vare(K^{\pm1})1$,
$E\ang 1=0=\vare(E)1$ and
$F\ang 1=0=\vare(F)1$, thus \rf[modeins] holds.
Let $f,g\in\ld$. Then
\begin{align*}
&K^{\pm1}\ang (fg)=Q^{\pm1}fgQ^{\mp1}=
Q^{\pm1}fQ^{\mp1}Q^{\pm1}gQ^{\mp1}
                =(K^{\pm1}\ang f)(K^{\pm1}\ang g),\\
&(E\ang f)g+(K\ang f)(E\ang g)=
              (Af-QfQ^{-1}A)g +QfQ^{-1}(Ag-QgQ^{-1}A)\\
&\phantom{(E\ang f)g+(K\ang f)(E\ang g)}
= Afg-QfgQ^{-1}A = E\ang (fg),
\end{align*}
and analogously $(F\ang f)(K^{-1}\ang g)+f(F\ang g)=F\ang (fg)$.
As $Q^*=Q$, we get
$$
(K^{\pm1}\ang f)^*=(Q^{\pm1}fQ^{\mp1})^*=
Q^{\pm1}f^*Q^{\mp1}=K^{\pm1}\ang f^*=
S(K^{\pm1})^*\ang f^*.
$$
Note that also $A^*=A$ and $B^*=B$ since
$\bar q=q^{-1}$ and $\im\lambda^{-1}\in\R$. Thus
\begin{align*}
 (E\ang f)^*&= f^*A-AQ^{-1}f^*Q=-EK^{-1}\ang f^* = S(E)^*\ang f^*,\\
 (F\ang f)^*&=Qf^*B-q^{-2}BQf^*=Q(q^2f^*QB-Bf^*Q)Q^{-1}
            =-KF\ang f^*\\
            &=S(F)^*\ang f^*.
\end{align*}
This proves \rf[modalg]. Therefore the $\slr$-action defined by
\rf[qhact1]--\rf[qhact3] turns $\ld$ into
a $\slr$-module *-algebra.

To demonstrate that Equations \rf[qhact1]--\rf[qhact3] define
an operator expansion of the action~$\ang$ on $\qh$, it now suffices to
verify it
for the generators $x$ and $y$. The result follows then by
applying the first relation of \rf[modalg]. Using \rf[xy], \rf[Qxy] and \rf[AB], we
obtain
\begin{align*}
K^{\pm1}\ang y &=Q^{\pm1}yQ^{\mp1}=q^{\pm2}y,\quad
K^{\pm1}\ang x=Q^{\pm1}xQ^{\mp1}=q^{\mp2}x,\\
E\ang y &=Ay-QyQ^{-1}A=-\im \lambda^{-1}(y^2-q^2y^2)=\im qy^2,\\
E\ang x &=Ax-QxQ^{-1}A=-\im q^{-2}\lambda^{-1}(q^2yx-xy)=-\im q^{-1},
\end{align*}
and similarly $ F\ang y=\im$,\, $F\ang x=-\im q^{2}x^{2}$.
\end{proof}

The aim of this section is to define an invariant integral
on an appropriate class of operators. The problem arises because
there does not exist a normalized invariant integral on
$\qh$. This can be seen as follows:
If there were a $\slr$-invariant functional $h$ on $\qh$ satisfying
$h(1)=1$, then
\begin{equation}
 1=h(1)=-\im h(F\ang y)=-\im \vare(F) h(y)       \label{1=0}
\end{equation}
would contradict $\vare(F)=0$.

In \cite{KW}, the quantum trace formula
$\tr_q(X):=\tr (XK^{-1})$
was generalized
by replacing $K^{-1}$ with the operator
which realizes the operator expansion of $K^{-1}$ on $\ld$.
The element $X$ should belong to an algebra which has the property
that the traces taken on the Hilbert space $\Hh=\bar D$ exist.
Two algebras with this property are described in
Equations \eqref{BA} and \eqref{FD}.
The following proposition shows that the generalized
quantum trace formula does define an invariant integral
on these algebras.
\begin{thp}                                       \label{qh-haar}
Suppose that $\pi:\qh\rightarrow \ld$ is a *-rep\-re\-sen\-ta\-tion of $\qh$
such that $Q^{-1}\in\ld$. Let $\fA$ be the O*-algebra generated
by the elements of $\pi(\qh)$ and $Q^{-1}$.
Then the *-algebras $\BBA$ and $\FFD$ defined in
\rf[BA] and \rf[FD], respectively, are $\slr$-module *-algebras,
where the action is given by \rf[qhact1]--\rf[qhact3].
The linear functional
\begin{equation}                                  \label{qh-h}
    h(g):=c\, \tr (\ov{gQ^{-1}}),\quad c\in\R,
\end{equation}
defines an invariant integral on both $\BBA$ and $\FFD$.
\end{thp}
\begin{proof}
It is obvious from the definitions of $\BBA$ and $\FFD$
that both algebras are stable under the $\slr$-action
given by \rf[qhact1]--\rf[qhact3]. Moreover, by Lemma \ref{qh-act},
$\BBA$ and $\FFD$ are $\slr$-module *-algebras.

Since the action is associative
and $\vare$ is a homomorphism, it suffices to
show the invariance for generators. Let $g\in \BBA$.
It follows from  \cite[Corollary 5.1.14]{S} that
$\tr(\ov{agb})=\tr(\ov{gba})=\tr(\ov{bag})$ for
all $a,b\in\fA$. Hence
$$
h(K^{\pm1}\ang g)=\tr(\ov{Q^{\pm1}gQ^{\mp1}Q^{-1}})=
\tr(\ov{gQ^{-1}})=\vare(K^{\pm1})h(g),
$$
$$
h(E\ang g)= \tr(\ov{AgQ^{-1}}\hsp-\hsp\ov{QgQ^{-1}AQ^{-1}})=
\tr(\ov{AgQ^{-1}})-\tr(\ov{AgQ^{-1}})=0=\vare(E)h(g),
$$
$$
h(F\ang g)= q\tr(\ov{Bg}-\ov{q^2gQBQ^{-1}})=
q\hs \tr(\ov{Bg})-\tr(\ov{Bg})=0=\vare(F)h(g),
$$
where we used $QB=q^{-2}BQ$ in the last line.
This proves the assertion for $\BBA$. Since $\FFD\subset \BBA$,
the same arguments apply to $\FFD$.
\end{proof}

Note that Proposition \ref{qh-haar} was proved without referring explicitly
to the *-rep\-re\-sen\-ta\-tion $\pi:\qh\rightarrow \ld$.
The only assumption on $\pi$ was that $Q^{-1}\in\ld$.
However, it is a priori not clear whether such a
rep\-re\-sen\-ta\-tion exists. That the answer to this question is affirmative
was shown in \cite{S3}. For the convenience of the reader, we
summarize the results from \cite{S3} (see also \cite{S4}).

First we introduce some notations.
Write $q=e^{\im\varphi}$ with $|\varphi|<\pi$ and $s(\varphi)$
for the sign of $\varphi$. Consider the Pauli matrices
\[ \sigma_0:= \left(
      \begin{array}{cc}    1     &    0 \\
                           0     &   -1    \end{array} \right),\quad
\sigma_1:= \left(
      \begin{array}{cc}    0     &    1 \\
                           1     &    0    \end{array} \right).      \]
Let $\K$ be a Hilbert space. The Pauli matrices act on
$\K\oplus\K$ in an obvious way.
We denote by $T$ and $P$ the multiplication operator by the variable
$t$ and the differential operator $\im\frac{\dd}{\dd t}$
acting on $\mathcal{L}^2(\R)$, respectively. An operator $\omega$ on $\K$
is called a symmetry if $\omega$ is unitary and self-adjoint.

For the definition of an integrable (well-behaved) *-rep\-re\-sen\-ta\-tion,
we refer the reader to \cite{S3} (see also \cite{S5}).
By \cite[Theorem 3.7]{S3}, each integrable *-rep\-re\-sen\-ta\-tion of $\qh$
such that $x$ and $y$ are self-adjoint operators and that
$\ker Q=\{0\}$ is unitarily equivalent to a rep\-re\-sen\-ta\-tion
which is given by one of the following models.
\begin{itemize}
\item[$(I)$:] $y=e^T\otimes\omega$,\ \,
$x=((-1)^kqe^{2(\varphi-k\pi)P}+1)e^{-T}\otimes \omega$\\
on $\Hh=\mathcal{L}^2(\R)\otimes\K$, where $\omega$ is a
symmetry on $\K$ and $k\in\{0,s(\varphi)\}$.
\item[$(II)$:] $y=e^T\otimes\sigma_1$,\ \,
$x=qe^{(2\varphi-s(\varphi)\pi)P}e^{-T}\otimes\sigma_0\sigma_1
        +e^{-T}\otimes \sigma_1$\\
on $\Hh=\mathcal{L}^2(\R)\otimes(\K\oplus\K)$.
\end{itemize}
The operator $Q$ and its inverse are given by
\begin{itemize}
\item[$(I)$:] $Q=(-1)^{k+1}e^{2(\varphi-k\pi)P}\otimes 1$,\ \,
        $Q^{-1}=(-1)^{k+1}e^{-2(\varphi-k\pi)P}\otimes 1$,
\item[$(II)$:]  $Q=-e^{(2\varphi-s(\varphi)\pi)P}\otimes\sigma_0$,\ \,
 $Q^{-1}=-e^{-(2\varphi-s(\varphi)\pi)P}\otimes\sigma_0$.
\end{itemize}

There exists a dense linear space $D\subset\Hh$ such that $D$
is an invariant core
for each of the self-adjoint operators $x$, $y$, $Q$ and $Q^{-1}$
and the commutation relation of these operators are
pointwise satisfied on $D$.
For instance, set
\begin{equation}                                    \label{F}
\F:=\Lin \{e^{-\epsilon t^2+\gamma t}\,;\,\epsilon>0,\
\gamma\in\C\},
\end{equation}
and take $D=\F\otimes\K$ and $D=\F\otimes(\K\oplus\K)$
for rep\-re\-sen\-ta\-tions
of type $(I)$ and $(II)$, respectively,
(see \cite[Proposition 2]{S4}). This proves, in particular, the existence
of rep\-re\-sen\-ta\-tions which satisfy the assumptions of
Proposition \ref{qh-haar}.

Motivated by a similar result in \cite{KW}, we view $\BBA$ as
the algebra of infinitely differentiable functions which vanish
sufficiently rapidly at ``infinity'' and $\FFD$ as the
infinitely differentiable functions with compact support.

Note that the rep\-re\-sen\-ta\-tion theory of $\qh$ is much more subtle
in comparison with the quantum disc treated in \cite{KW}.
Whereas for the quantum disc one can define algebras
of  functions which vanish sufficiently rapidly at ``infinity''
and an invariant integral on purely algebraic level
(see \cite[Lemma 3.4]{KW} and \cite[Proposition 3.3]{KW}), the same method
does not apply to $\qh$. For instance, since the
self-adjoint operator $Q$ has
continuous spectrum, there does not exist a non-zero
continuous function $\psi$ on $\sigma(Q)\,(=\R)$ such that
$\psi(Q)Q^{-1}$ is of trace class. Furthermore, apart from
the trivial rep\-re\-sen\-ta\-tion $x=\alpha$, $y=\alpha^{-1}$, where
$\alpha\in\R\backslash \{0\}$, the algebra $\qh$ does not admit other 
(irreducible) bounded *-rep\-re\-sen\-ta\-tions.
Nevertheless, we succeeded in establishing an invariant
integration theory on $\qh$.

\section{Real \boldmath{$q$}-Weyl algebra}
                                                           \label{sec-qhn}

Recall that  $|q|=1$, $q^4\neq1$. Let $n\in\N$.
The *-algebra $\qhn$ with hermitian generators
$x_1,\ldots,x_n$, $y_1,\ldots,y_n$ and relations
\begin{eqnarray}
  y_k y_l &=& q y_l y_k,\quad k<l,       \label{qhn0} \\
  x_k x_l &=& q^{-1} x_l x_k,\quad k<l,       \label{qhn1} \\
  x_l y_k &=& q y_k x_l,\quad k\neq l, \label{qhn2}\\
  x_k y_k &=& q^2 y_k x_k                 \label{qhn3}
  -(1-q^2)\sum^n_{j=k+1}q^{j-k}y_j x_j+(1-q^2)q^{n-k},\quad k<n,\\
 x_n y_n &=& q^2 y_n x_n + (1-q^2)        \label{qhn4}
\end{eqnarray}
is called {\it real $q$-Weyl algebra} \cite{S4}.

Define the hermitian elements
\begin{equation}                              \label{def-Q}
    Q_j=\lambda^{-1}(y_jx_j-x_jy_j),\,\ j\leq n,\,\quad
  Q_{n+1}=1.
\end{equation}
A straightforward calculation shows that
\begin{equation}                              \label{Qy}
 Q_ky_j=y_jQ_k,\,\ j<k,\quad \, Q_ky_j=q^2y_jQ_k,\,\ j\geq k,
\end{equation}
\begin{equation}                              \label{Qx}
Q_kx_j=x_jQ_k,\,\ j<k,\quad\,Q_kx_j=q^{-2}x_jQ_k,\,\ j\geq k.
\end{equation}
This immediately implies
\begin{equation}                              \label{QQn}
              Q_kQ_l=Q_lQ_k ,\,\quad \,
\mathrm{for\ all}\ \, k,l\leq n+1.
\end{equation}
From the definition of $Q_k$ and \rf[qhn3], we obtain
\begin{align*}
Q_k&=\lambda^{-1} (y_kx_k-x_ky_k)=\lambda^{-1}(1-q^2)
\big(y_kx_k+ \mbox{$\sum$}^n_{j=k+1}q^{j-k}y_j x_j-q^{n-k}\big)\\
  &=q\big(q^{n-k}- \mbox{$\sum$}^n_{j=k}q^{j-k}y_j x_j\big),
\end{align*}
so that
\begin{align*}
q^{-1}Q_k&=q^{n-k}\hsp - \hsp \mbox{$\sum$}^n_{j=k}q^{j-k}y_j x_j
   =-y_kx_k+q\big(q^{n-(k+1)}\hsp - \hsp
\mbox{$\sum$}^n_{j=k+1}q^{j-(k+1)}y_j x_j\big)\\
    &=-y_kx_k+Q_{k+1}.
\end{align*}
Hence
\begin{equation}                      \label{QQyx}
y_kx_k=(Q_{k+1}-q^{-1}Q_k),\,\quad
x_ky_k=(Q_{k+1}-qQ_k),
\end{equation}
where the second equation follows from the first one by taking
adjoints. Equation \rf[QQyx] also holds for $k=n$. Furthermore,
from \rf[QQyx],
\begin{equation}                      \label{xyQ}
   x_ky_k-q^2y_kx_k=(1-q^2)Q_{k+1}.
\end{equation}

As in Section \ref{sec-qh}, write $q=e^{\im\varphi}$ with
$|\varphi|<\pi$ and set $q_0:=e^{\im\varphi/2}$.
Consider the $\slrn$-action on $\qhn$ defined by
\begin{align}
&j<n:                                          \label{Anfang}
& & E_j\ang y_k=0,& & k\neq j+1,
 & &  E_j\ang y_{j+1}=-\im q_0^{-1}y_j,\\
&  &  &E_j\ang x_k=0,& & k\neq j,
 & &  E_j\ang x_{j}=\im q_0^{-1}x_{j+1},     \label{Ang3} \\
&  &  &F_j\ang y_k=0, & & k\neq j,
& &  F_j\ang y_{j}=\im q_0y_{j+1},\\
&  &  &F_j\ang x_k=0,  & & k\neq j+1,
 & &  F_j\ang x_{j+1}=-\im q_0x_{j},                       \label{Ang4}\\
&K_j\ang y_k=y_k, & &k\neq j,j+1, & &K_j\ang y_j=qy_j,
 & &  K_j\ang y_{j+1} =q^{-1}y_{j+1},                    \label{Ang1} \\
&K_j\ang x_k=x_k, & &k\neq j,j+1, & &K_j\ang x_j=q^{-1}x_j,
 & &   K_j\ang x_{j+1} =qx_{j+1},                          \label{Ang2}
\end{align}
\begin{align}
&j=n:\qquad \  & &  E_n\ang y_k=\im q y_ny_k, & &\!\! \!\! k<n, \ \ \,\,
& & E_n\ang y_n =\im q y_n^2,                             \label{AngEj}\\
& & & E_n\ang x_k=0, & & \!\!\!\! k<n, \ \ \,\,
& &  E_n\ang x_n =-\im q^{-1},                          \label{Ang7}\\
& & & F_n\ang y_k=0, & & \!\!\!\! k<n,\ \ \,\,
& & F_n\ang y_n = \im,  \\
& & & F_n\ang x_k=-\im q^{2}x_kx_n, & &\!\!\!\! k<n,\ \ \,\,
& & F_n\ang x_n =-\im q^{2}x^{2}_n, \quad \\
& & & K_n\ang y_k=qy_k, & &\!\!\!\!  k<n, \ \ \,\,
& & K_n\ang y_n =q^2y_n,       \label{Ang5}\\
& & & K_n\ang x_k=q^{-1}x_k, & & \!\!\!\! k<n, \ \ \,\,
& & K_n\ang x_n =q^{-2}x_n,                               \label{Ende}
\end{align}
Instead of proving that these formulas define a $\slrn$-action on
$\qhn$, we give an operator expansion of the action showing
thus its existence  since the algebra admits faithful
*-rep\-re\-sen\-ta\-tions. More explicitly, if
$\pi:\qhn\rightarrow\ld$ is a faithful *-rep\-re\-sen\-ta\-tion and
$\pi(\qhn)$ a $\slrn$-module *-algebra with action $\blacktriangleright$,
then setting $X\ang f:=\pi^{-1}(X\blacktriangleright\pi(f))$,
$f\in\qhn$, $X\in\slrn$,
defines a $\slrn$-action on $\qhn$ turning it into a
$\slrn$-module *-algebra.
According to our notational convention, we do not
designate explicitly the rep\-re\-sen\-ta\-tion. If that causes confusion,
one can always think of the elements of $\qhn$ as operators
in $\ld$.

Remind that we only consider ``integrable'' 
rep\-re\-sen\-ta\-tions which are in a certain
sense well-behaved. Integrable rep\-re\-sen\-ta\-tions of $\qhn$ were
defined and classified in \cite{S4}. 
Before describing an operator expansion of the action, we state the
results from \cite{S4}. This allows us to perform any algebraic
manipulation in $\ld$ and to verify directly if certain
assumptions on the operators are satisfied.

We shall adopt the notational conventions of
Section \ref{sec-qh}. In particular,
$T$ and $P$ denote the multiplication operator by the variable
$t$ and the differential operator $\im\frac{\dd}{\dd t}$
acting on $\mathcal{L}^2(\R)$, respectively.
Set $\Hh=\cL^2(\R)^{\otimes n}\otimes \K$, where $\K$ is a
Hilbert space.
Consider the following series of Hilbert space operators
on $\Hh$.
\begin{align*}
&(I):& y_l&=(\otimes_{j=1}^{n-l}e^{(\varphi-k_{n+1-j}\pi)P})
\otimes e^T\otimes1\otimes\cdots\otimes1\otimes\omega_l, \\
& & x_l&=(\otimes_{j=1}^{n-l}(-1)^{(k_{n+1-j}+1)}e^{(\varphi-k_{n+1-j}\pi)P})\\
& & & \hspace{90pt}\otimes ((-1)^{k_l}qe^{2(\varphi-k_l\pi)P}+1)e^{-T}
\otimes1\otimes\cdots\otimes1\otimes \omega_l
\end{align*}
for all $l=1,\ldots,n$, where $k_j\in\{0,s(\varphi)\}$ and
the operators $\omega_j$ are symmetries on $\K$ satisfying
$\omega_j\omega_l=(-1)^{k_j}\omega_l\omega_j$ for $j> l$.
\begin{align*}
&(II): & y_l&=(\otimes_{j=1}^{n-l}e^{(\varphi-k_{n+1-j}\pi)P})
\otimes e^T\otimes1\otimes\cdots\otimes1\otimes\omega_l,\\
& & x_l&=(\otimes_{j=1}^{n-l}(-1)^{(k_{n+1-j}+1)}e^{(\varphi-k_{n+1-j}\pi)P})\\
& & & \hspace{80pt}\otimes ((-1)^{k_l}qe^{2(\varphi-k_l\pi)P}+1)e^{-T}
\otimes1\otimes\cdots\otimes1\otimes \omega_l
\end{align*}
for $l=2,\ldots,n$, and
\begin{align*}
y_1&=(\otimes_{j=1}^{n-1}e^{(\varphi-k_{n+1-j}\pi)P})\otimes
  e^T\otimes\omega_1,\\
x_1&=(\otimes_{j=1}^{n-1}(-1)^{(k_{n+1-j}+1)}e^{(\varphi-k_{n+1-j}\pi)P})\\
    &\hspace{40pt}\otimes (e^{-T}\otimes \omega_1+(-1)^{(k_n+1)+\ldots+(k_{2}+1)}
 qe^{(2\varphi-s(\varphi)\pi)P}e^{-T}\otimes\omega_0\omega_1),
\end{align*}
where $k_j\in\{0,s(\varphi)\}$,\,\ $j=2,\ldots,n$, and
the operators $\omega_j$ are symmetries on $\K$ satisfying
$\omega_j\omega_l=(-1)^{k_j}\omega_l\omega_j$ for 
$j> l$,\,\ $\omega_j\omega_0=\omega_0\omega_j$ for $j\ge 2$
and $\omega_1\omega_0=-\omega_0\omega_1$.

The proof of the following facts can be found in \cite{S4}.
\begin{thp}                                      \label{Darst}
 \renewcommand{\labelenumi}{\roman{enumi}}
\begin{enumerate}
\item Both families of operators $x_1,\ldots,x_n$, $y_1,\ldots,y_n$
      define an integrable *-rep\-re\-sen\-ta\-tion of $\qhn$.
\item Any integrable *-rep\-re\-sen\-ta\-tion of $\qhn$ such that
      ${\rm ker}~Q_j=\{0\}$ for all $j=1,\ldots,n$ is unitarily
      equivalent to one of the above form.
\item There exists a dense domain $D$ of $\Hh$ such that $D$ is an
      invariant core for each of the self-adjoint operators
      $x_j$, $y_j$, $Q_j$, $j=1,\ldots,n$, and the commutation relations
      in $\qhn$ are pointwise fulfilled on $D$, for instance,
      $D:=\F^{\otimes n}\otimes\K$ satisfies this conditions,
      where $\F$ is defined as in \eqref{F}.
\item With $D$ given as in (iii), the *-rep\-re\-sen\-ta\-tion
      of $\qhn$ into $\ld$ is faithful.
\end{enumerate}
\end{thp}

In the remainder of this section we shall exclusively
work with rep\-re\-sen\-ta\-tions of the series $(I)$ and assume
$k_n=\ldots=k_1=0$.
It follows from the above formulas that
in this case the operators $Q_l$, $Q_l^{-1}$, $|Q_l|^{1/2}$ and
$|Q_l|^{-1/2}$ are given by
\begin{equation}                                 \label{QQ-1}
Q_l=(-1)^{n-l+1}(\mathop{\otimes}_{j\ge l}e^{2\varphi P})
     \otimes1\otimes\cdots\otimes1,\ \,
Q_l^{-1}=(-1)^{n-l+1}(\mathop{\otimes}_{j\ge l}e^{-2\varphi P})
     \otimes1\otimes\cdots\otimes1,
\end{equation}
\begin{equation*}                                 
|Q_l|^{1/2}=(\mathop{\otimes}_{j\ge l}e^{\varphi P})
     \otimes1\otimes\cdots\otimes1, \quad
|Q_l|^{-1/2}=(\mathop{\otimes}_{j\ge l}e^{-\varphi P})
     \otimes1\otimes\cdots\otimes1.\qquad\quad\mbox{ }
\end{equation*}
By a slight abuse of notation, we denote operators 
and their restrictions to a dense domain $D$ by the same symbol. 

\begin{thc}                                 \label{Q-rep}
Suppose that the operators $x_j$, $y_j$, $j=1,\ldots,n$, are given
by the formulas of the series $(I)$ and assume that
$k_n=\ldots=k_1=0$.
Then there exists a dense domain $D$ of $\Hh$ such that $D$ is an
invariant core for the self-adjoint operators
$x_j$, $y_j$, $Q_j$, $j=1,\ldots,n$, the commutation relations
in $\qhn$ are pointwise fulfilled on $D$,
the *-rep\-re\-sen\-ta\-tion $\pi:\qhn\rightarrow \ld$ is faithful,
and $|Q_j|^{-1/2}\in\ld$.
The operators $x_j$, $y_j$ and $|Q_j|^{1/2}$ satisfy the
following commutation relations.

\begin{align}       \label{Q12y}
\begin{split}
  &|Q_k|^{1/2}y_j  =y_j|Q_k|^{1/2},\ \ j<k, \quad 
  |Q_k|^{1/2}y_j  =q y_j|Q_k|^{1/2},\ \ j\geq k,\\
&|Q_k|^{1/2}x_j  =x_j|Q_k|^{1/2},\ \ j<k, \quad 
|Q_k|^{1/2}x_j  =q^{-1}x_j|Q_k|^{1/2},\ \ j\geq k.
\end{split}
\end{align}
\end{thc}
\begin{proof}
The assertions concerning $x_j$, $y_j$ and $Q_j$ are just
a repetition of Proposition \ref{Darst}.
Recall from Section \ref{sec-qh} that
$\F:=\Lin \{e^{-\epsilon t^2+\gamma t}\,;\,\epsilon>0,\
\gamma\in\C\}$ and $q=e^{\im \varphi}$ with $|\varphi|<\pi$.
The operators $e^{\alpha T}$ and $e^{\beta P}$, $\alpha,\beta\in\R$,
act on $e^{-\epsilon t^2+\gamma t}\in \F$ by (see \cite[Lemma 1.1]{S3})
$$
e^{\alpha T}(e^{-\epsilon t^2+\gamma t})
=e^{-\epsilon t^2+(\gamma+\alpha)t},\quad
e^{\beta P}(e^{-\epsilon t^2+\gamma t})
=e^{-\epsilon (t+\im \beta)^2+\gamma (t+\im\beta)}.
$$
Obviously, $\F$ is invariant under the action of $e^{\alpha T}$ and
$e^{\beta P}$. Hence $D:=\F^{\otimes n}\otimes\K$ satisfies the
conditions of the corollary.
On $\F$, the operators $e^{\alpha T}$ and $e^{\beta P}$ obey the
commutation relation
\begin{equation}                                       \label{TP}
e^{\beta P}e^{\alpha T}=e^{\im\beta\alpha}e^{\alpha T}e^{\beta P}.
\end{equation}
Now \rf[Q12y] is easily proved by
inserting the expressions of $y_j$, $x_j$, $|Q_j|^{1/2}$ and
applying \rf[TP] since $e^{\im \varphi}=q$.
\end{proof} 
\begin{thl}                                   \label{l-ABrho}
Suppose we are given an integrable *-rep\-re\-sen\-ta\-tion  of
$\qhn$ such that the operators $x_j$, $y_j$
and the domain $D$ satisfy the conditions of
Corollary \ref{Q-rep}. Set $q_0:=e^{\im \varphi/2}$, where
$|\varphi|<\pi$ and $q= e^{\im \varphi}$.
Define
\begin{align}
 \rho_k &= |Q_k|^{1/2} |Q_{k+1}|^{-1}|Q_{k+2}|^{1/2},
 &\!\!\! k&<n, \,\,\, & \rho_n&=|Q_1|^{1/2}|Q_n|^{1/2},         \label{qhnrho}\\
 A_k&=\im \lambda^{-1}q_0^{-1}q^{-1}Q_{k+1}^{-1}x_{k+1}y_k,
 &\!\!\! k&<n,\,\,\, & A_n&= -\im \lambda^{-1}y_n,               \label{qhnA}\\ 
 B_k&=-\im \lambda^{-1}q_0\rho_k^{-1}Q_{k+1}^{-1}y_{k+1}x_k,
 &\!\!\! k&<n,\,\,\, & B_n&= -\im \lambda^{-1}q^{-1}\rho_n^{-1}x_n. \label{qhnB}
\end{align}
Then the operators $\rho_k$, $A_k$, $B_k$ are hermitian and they
obey the following commutation relations:
\begin{equation}                                 \label{qAB1}
\rho_i\rho_j=\rho_j\rho_i,\ \ \rho_j^{-1}\rho_j=\rho_j\rho_j^{-1}=1,
\ \ \rho_iA_j=q^{a_{ij}}A_j\rho_i,\ \ \rho_iB_j=q^{-a_{ij}}B_j\rho_i,
\end{equation}
\begin{equation}                                 \label{qAB2}
A_iA_j-A_jA_i=0,\ \, i\neq j\pm1,
\quad\,
A_j^2A_{j\pm1}-(q+q^{-1})A_jA_{j\pm1}A_j+A_{j\pm1}A_j^2=0,
\end{equation}
\begin{equation}                                \label{qAB3}
B_iB_j  -  B_jB_i  =  0,\ \,i  \neq   j\pm1,
\quad B_j^2B_{j\pm1}  -  (q  +  q^{-1})B_jB_{j\pm1}B_j  
+  B_{j\pm1}B_j^2  =  0,
\end{equation}
\begin{equation}                                \label{qAB4}
  A_iB_j-A_jB_i=0,\ \,\, i\neq j,  \quad\quad
  A_jB_j-B_jA_j=\lambda^{-1}
  (\rho_j-\rho_j^{-1}),\,\  j<n,
\end{equation}
\begin{equation}                               \label{qAB5}
         A_nB_n -B_nA_n \, = \,   -\lambda^{-1}\rho_n^{-1},
\end{equation}
where $(a_{ij})_{i,j=1}^n$ denotes the Cartan matrix of
$\mathrm{sl}(n+1,\C)$.
\end{thl}
\begin{proof}
Clearly, the operators $\rho_k$ are hermitian. It follows from
\rf[Q12y] that
\begin{align}                               \label{rhoy}
&\rho_j  y_k\!=\!y_k\rho_j,\,\ k\!\neq\! j,j\!+\!1,\,\quad \rho_j  y_j\!=\!qy_j\rho_j,\,\quad
\rho_j  y_{j+1} \!=\!q^{-1}y_{j+1}\rho_j,\\
                                                         \label{rhox}
&\rho_j  x_k\!=\!x_k\rho_j,\,\ k\!\neq\! j,j\!+\!1,\,\quad \rho_j  x_j\!=\!q^{-1}x_j\rho_j,
\,\quad \rho_j  x_{j+1} \!=\!qx_{j+1}\rho_j,\\
                                                        \label{rhon}
&\rho_n  y_k\!=\!qy_k\rho_n,\ \rho_n  x_k\!=\!q^{-1}x_k\rho_n,\,\ k\!<\!n,\quad
\rho_n  y_n \!=\!q^2y_n\rho_n,\ \rho_n  x_n \!=\!q^{-2}x_n\rho_n.
\end{align}
Observe that $\im\lambda\in\R$ and $q=q_0^2$. Thus,
by \rf[qhn2], \rf[Qy], \rf[Qx] and the preceding, we have: 
\begin{align*}
 &k\neq n: & &
 A_k^*=\im\lambda^{-1} q_0q y_kx_{k+1}Q_{k+1}^{-1}=
\im \lambda^{-1}q_0q^{-2}Q_{k+1}^{-1}x_{k+1}y_k=A_k, \\
 & & & B_k^*=-\im \lambda^{-1}q_0^{-1}x_ky_{k+1}Q_{k+1}^{-1}\rho_k^{-1}
 =-\im \lambda^{-1}q_0^{-1}q\rho_k^{-1}Q_{k+1}^{-1}y_{k+1}x_k=B_k,\\
 &k=n: & & 
 A_n^*= -\im \lambda^{-1}y_n=A_n,\quad
B_n^*= -\im \lambda^{-1}qx_n\rho_n^{-1}
=-\im \lambda^{-1}q^{-1}\rho_n^{-1}x_n=B_n. 
\end{align*}

Equation \rf[qAB1] is easily shown using \rf[QQn]
and \rf[rhoy]--\rf[rhon].
The first relations of \rf[qAB2]--\rf[qAB4] follow by straightforward
computations using the commutation rules in $\qhn$ and
\rf[rhoy]--\rf[rhon].
Let $l<n$. 
Since $(-1)^{n-l+1}(-1)^{n-(l+2)+1}=1$, 
we find from Equations \rf[QQ-1] and \rf[qhnrho] that 
$\rho_l^{-1}Q_lQ_{l+1}^{-2}Q_{l+2}=\rho_l$.
Thus 
\begin{align*}
&A_kB_k-B_kA_k
 = \lambda^{-2}\rho_k^{-1}Q_{k+1}^{-2}
                     (x_{l+1}y_{l+1}y_lx_l-y_{l+1}x_{l+1}x_ly_l)\\
& = \lambda^{-2}\rho_k^{-1}Q_{k+1}^{-2}[(Q_{l+2}\!-\!qQ_{l+1})
(Q_{l+1}\!-\!q^{-1}Q_{l})\!-\!(Q_{l+2}\!-\!q^{-1}Q_{l+1})(Q_{l+1}\!-\!qQ_{l})]\\
& = \lambda^{-2}\rho_k^{-1}Q_{k+1}^{-2}
[(q-q^{-1})Q_{l+2}Q_l-(q-q^{-1})Q_{l+1}^2]
= \lambda^{-1} (\rho_k -\rho_k^{-1}),
\end{align*}
where we applied \rf[QQyx] in the
second equality. 
The proof of \rf[qAB5] is similar and easier. 

To verify the second equations of \rf[qAB2] and \rf[qAB3],
we first observe that
\begin{eqnarray}
A_{k-1}A_k &=& qA_kA_{k-1}+\lambda^{-1}q^{-1}Q_k^{-1}x_{k+1}y_{k-1},
 \quad \,  k<n,                                    \label{AAk}\\
B_{k-1}B_k &=& q^{-1}B_kB_{k-1}
-\lambda^{-1}\rho_{k-1}^{-1}\rho_k^{-1}Q_k^{-1}y_{k+1}x_{k-1},
 \quad \,  k<n,                                   \label{BBk}  \\
A_{n-1}A_n&=&qA_nA_{n-1}-\lambda^{-1}q_0Q_n^{-1}y_{n-1},
                                               \label{AAn-1}   \\
B_{n-1}B_n&=&q^{-1}B_nB_{n-1}-                  \label{BBn}
   \lambda^{-1}q_0^{-1}q^{-1}\rho_{n-1}^{-1}\rho_n^{-1}Q_n^{-1}x_{n-1}.
\end{eqnarray}
To see this, consider
\begin{align*}
&\rho_{k-1}^{-1}Q_{k}^{-1}y_{k}x_{k-1}\rho_{k}^{-1}Q_{k+1}^{-1}y_{k+1}x_k
=\rho_{k}^{-1} Q_{k+1}^{-1}y_{k+1}y_{k}x_k
\rho_{k-1}^{-1}Q_{k}^{-1}x_{k-1}\\
&\hspace{60pt}=\rho_{k}^{-1}Q_{k+1}^{-1}y_{k+1}(q^{-2}x_ky_{k}-q^{-2}(1-q^2)Q_{k+1})
                                \rho_{k-1}^{-1}Q_{k}^{-1}x_{k-1}\\
&\hspace{60pt}=q^{-1}Q_{k+1}^{-1}y_{k+1}x_kQ_{k}^{-1}y_{k}x_{k-1}
+\lambda q^{-1}Q_{k}^{-1}y_{k+1}x_{k-1},
\end{align*}
where the second equality was obtained by inserting
Equation \rf[xyQ].
Multiplying both sides by
$(-\im \lambda^{-1}q_0)^2$
gives \rf[BBk] since
 $(-\im \lambda^{-1}q_0)^2\lambda q^{-1}=-\lambda^{-1}$.
Equations \rf[AAk], \rf[AAn-1] and \rf[BBn] are proved similarly.

Next we claim that
\begin{eqnarray}
A_kQ_k^{-1}x_{k+1}y_{k-1}&=&qQ_k^{-1}x_{k+1}y_{k-1}A_k,\label{hilf1}\\
 A_{k-1}Q_k^{-1}x_{k+1}y_{k-1}&=&q^{-1}Q_k^{-1}x_{k+1}y_{k-1}A_{k-1},
                                                      \label{hilf2}\\
 B_k\rho_{k-1}^{-1}\rho_k^{-1}Q_k^{-1}y_{k+1}x_{k-1}&=&
 q^{-1}  \rho_{k-1}^{-1}\rho_k^{-1}Q_k^{-1}y_{k+1}x_{k-1}B_k,  \\
   B_{k-1}\rho_{k-1}^{-1}\rho_k^{-1}Q_k^{-1}y_{k+1}x_{k-1}&=&
 q \rho_{k-1}^{-1}\rho_k^{-1}Q_k^{-1}y_{k+1}x_{k-1}B_{k-1}, \\
 A_nQ_n^{-1}y_{n-1}&=&qQ_n^{-1}y_{n-1}A_n, \\
 A_{n-1}Q_n^{-1}y_{n-1}&=&qQ_n^{-1}y_{n-1}A_{n-1}, \\
 B_n\rho_{n-1}^{-1}\rho_n^{-1}Q_n^{-1}x_{n-1}
&=&q^{-1}\rho_{n-1}^{-1}\rho_n^{-1}Q_n^{-1}x_{n-1}B_n,  \\
 B_{n-1}\rho_{n-1}^{-1}\rho_n^{-1}Q_n^{-1}x_{n-1}
&=&q\rho_{n-1}^{-1}\rho_n^{-1}Q_n^{-1}x_{n-1}B_{n-1}. \label{hilf8}
\end{eqnarray}
All these equations are easily shown by straightforward calculations.
As a sample,
\begin{align*}
&Q_{k+1}^{-1}x_{k+1}y_kQ_k^{-1}x_{k+1}y_{k-1}
=Q_k^{-1}Q_{k+1}^{-1}x_{k+1}y_kx_{k+1}y_{k-1}\\
&\hspace{60pt}=qQ_k^{-1}x_{k+1}Q_{k+1}^{-1}x_{k+1}y_ky_{k-1}
=qQ_k^{-1}x_{k+1}y_{k-1}Q_{k+1}^{-1}x_{k+1}y_k
\end{align*}
implies \rf[hilf1].

Now, the second equations of \rf[qAB2] and \rf[qAB3]
follow readily from \rf[AAk]--\rf[BBn] and \rf[hilf1]--\rf[hilf8].
For example, if $k<n$, then computing
$\rf[AAk]\cdot A_k-q^{-1}A_k \cdot \rf[AAk]$ and
applying \rf[hilf1] gives \rf[qAB2] with the plus sign, and
$A_{k-1} \cdot \rf[AAk]-q^{-1}\rf[AAk]\cdot A_{k-1}$ together
with \rf[hilf2] gives \rf[qAB2] with the minus sign.
By the same method one proves the remaining relations.
\end{proof}

We are now in a position to present the operator expansion
of the action announced in the beginning of this section.

\begin{thl}                                      \label{qhn-act}
Suppose we are given an integrable *-rep\-re\-sen\-ta\-tion  of
$\qhn$ such that the operators $x_j$, $y_j$ 
and the domain $D$ satisfy the conditions of
Corollary \ref{Q-rep}. With the  operators $\rho_k$, $A_k$ and $B_k$
defined in Lemma \ref{l-ABrho}, set
\begin{eqnarray}
  K_j\,\ang\, f &=&\rho_jf\rho_j^{-1},\ \quad K_j^{-1}\,\ang\, f
 \ =\ \rho_j^{-1}f\rho_j,                           \label{qhnact1}\\
  E_j\,\ang\, f &=& A_jf-\rho_jf\rho_j^{-1}A_j,     \label{qhnact2}\\
  F_j\,\ang\, f &=& B_jf\rho_j-q^2f\rho_jB_j      \label{qhnact3}
\end{eqnarray}
for $j=1,\ldots,n$.
Then Equations \rf[qhnact1]--\rf[qhnact3] define a $\slrn$-action
$\ang$ on $\ld$ turning 
$\ld$ into a $\slrn$-module *-algebra.
Its restriction to $\qhn$, considered as subalgebra of $\ld$, is given
by the formulas \rf[Anfang]--\rf[Ende].
\end{thl}
\begin{proof} The operators $\rho_j$, $A_j$, $B_j$ satisfy
the same commutation relations as the corresponding operators in 
\cite[Lemmas 4.1 and 4.2]{KW} 
(with $\epsilon_1=\ldots=\epsilon_n=1$ in \cite[Equation (64)]{KW}).
Since only these relations are needed in order to verify that 
\rf[qhnact1]--\rf[qhnact3] define a $\sltn$-action on $\ld$, 
the proof of this fact 
runs completely 
analogous to that of 
\cite[Lemmas 4.2]{KW}. 
By the same argument, the 
action satisfies \rf[modeins] and the first equation of \rf[modalg] 
since these relations are independent of the involution. 
It remains to verify that $\ang$ is
consistent with the second equation of \rf[modalg]. 
This follows immediately from the proof of Lemma \ref{qh-act} by replacing
$Q$, $A$ and $B$ with $\rho_j$, $A_j$ and $B_j$, respectively.

Applying \rf[Q12y], one easily proves by direct calculations that 
\rf[qhnact1] yields the action of $K_j^{\pm1}$
on $x_k$ and $y_k$,\,\ $j,k=1,\ldots,n$.

Next, using the commutation rules
of $x_i$, $y_i$ $Q_i$ and $\rho_i$, we obtain for $j<n$
\begin{align*}
  E_j\ang x_k &=\im \lambda^{-1}q_0^{-1}q^{-1}
(Q_{j+1}^{-1}x_{j+1}y_jx_k-\rho_jx_k\rho_j^{-1}Q_{j+1}^{-1}x_{j+1}y_j)\\
 &=\im \lambda^{-1}q_0^{-1}q^{-1}
(Q_{j+1}^{-1}x_{j+1}y_jx_k-Q_{j+1}^{-1}x_{j+1}y_jx_k) = 0,\quad k\neq j,   
\end{align*}
and, similarly, $E_j\ang y_k=0$ if $k\neq j+1$. Further, again for $j<n$, 
\begin{align*}
E_j\ang x_j &=\im \lambda^{-1}q_0^{-1}q^{-1}
(Q_{j+1}^{-1}x_{j+1}y_jx_j-\rho_jx_j\rho_j^{-1}Q_{j+1}^{-1}x_{j+1}y_j)\\
&=\im \lambda^{-1}q_0^{-1}q^{-1}Q_{j+1}^{-1}(q^2y_jx_j-x_jy_j)x_{j+1}
 = \im q_0^{-1}x_{j+1}, \\ 
 E_j\ang y_{j+1} &=\im \lambda^{-1}q_0^{-1}q^{-1}
(Q_{j+1}^{-1}x_{j+1}y_jy_{j+1}
   -\rho_jy_{j+1}\rho_j^{-1}Q_{j+1}^{-1}x_{j+1}y_j)\\
&=\im \lambda^{-1}q_0^{-1}Q_{j+1}^{-1}(x_{j+1}y_{j+1}-y_{j+1}x_{j+1})y_j
 = -\im q_0^{-1}y_{j}, 
\end{align*}
where we used \rf[def-Q] and \rf[xyQ]. 
If $k\neq n$, Equations \rf[qhn0], \rf[qhn2] and \rf[rhon] give 
\begin{align*}
E_n\ang x_k &=-\im \lambda^{-1}(y_nx_k-\rho_nx_k \rho_n^{-1}y_n)
      =-\im \lambda^{-1}(y_nx_k-y_nx_k)=0, \\
 E_n\ang y_k &=-\im \lambda^{-1}(y_ny_k-\rho_ny_k \rho_n^{-1}y_n)
            =-\im \lambda^{-1}(y_ny_k-q^2y_ny_k)=\im qy_ny_k.  
\end{align*}
The action of $E_n$ on $x_n$ and $y_n$ can be
computed by replacing in
the proof of Lemma \ref{qh-act} $x$, $y$, $\rho$ and  $A$ with 
$x_n$, $y_n$, $\rho_n$ and  $A_n$, respectively.

The preceding shows that the action of $E_i$, $i=1,\ldots,n$, 
on the generators of $\qhn$ is given by 
Equations \rf[Anfang], \rf[Ang3], \rf[AngEj] and \rf[Ang7]. 
The analogous statement for $F_i$, $i=1,\ldots,n$, is 
proved similarly. 
\end{proof}

Recall that a generalization of the quantum trace formula 
\cite[Section 7.1.6]{KS} was obtained in \cite{KW} 
by introducing the operator 
$\Gamma:=\prod_{l=1}^n \rho_l^{-l(n-l+1)}$. 
This definition resembles the definition of the distinguished 
element $K_0:=\prod_{l=1}^n K_l^{l(n-l+1)}\in\sltn$ 
satisfying $X K_{0}=K_{0}S^2(X)$ for all  $X\in\sltn$. 
Moreover, $K_0$ appears in the quantum trace as a density operator. 
This analogy will be used in the following proposition to define an 
invariant integral. Note that 
\begin{equation*}                                  
 \Gamma=|Q_1|^{-n}|Q_2|\cdots |Q_n|,\,\ n>1,\, \quad
 \Gamma=|Q_1|^{-1},\,\ n=1,
\end{equation*}
exactly as in \cite[Equation (71)]{KW}. 
\begin{thp}                                       \label{qhn-haar}
Suppose we are given a *-rep\-re\-sen\-ta\-tion  of
$\qhn$ into $\ld$ such that the operators $x_j$, $y_j$, $j=1,\ldots,n$
are given by the formulas of the series $(I)$ with
$k_n=\ldots=k_1=0$.
Assume that $D$ is of the form described in Corollary \ref{Q-rep}.
Let $\fA$ be the O*-algebra generated by the elements of
$\qhn\cup\{|Q_k|^{1/2},|Q_k|^{-1/2}\}_{k=1}^n$.
Then the *-algebras $\FFD$ and $\BBA$ defined in Equations 
\rf[FD] and \rf[BA], respectively, are $\slrn$-module *-algebras,
where the action is given by \rf[qhnact1]--\rf[qhnact3].
The linear functional
\begin{equation}                                  \label{BH}
    h(f):=c\, \tr( \ov{f\Gamma }),\quad c\in\R,
\end{equation}
defines an invariant integral on both $\FFD$ and $\BBA$.
\end{thp}
\begin{proof} Since the operators $\rho_k$, $A_k$, $B_k$ from
Lemma \ref{qhn-act}
and $\Gamma$ satisfy the same commutation relations as the 
corresponding operators in \cite{KW} (with $\epsilon_1=\ldots=\epsilon_n=1$ in \cite[Equation (64)]{KW}),
and since only these relations are needed in the proof of
\cite[Proposition 4.2]{KW}, the proof of Proposition \ref{qhn-haar}
is literally the same.
\end{proof}

Observe  that $z_n$, $z_n^*$, $K_n^{\pm1}$, $E_n$ and $F_n$ satisfy the
relations of $\qh$. In particular, Equation \rf[1=0] applies, 
telling us that we are dealing with a non-compact quantum space.
Again, $\BBA$ is considered as
the algebra of infinitely differentiable functions which vanish
sufficiently rapidly at ``infinity'' and $\FFD$ as the
infinitely differentiable functions with compact support.

Let us finally remark that, 
for $n\!>\!1$, the operators $K_j^{\pm1}$\!, $E_j$
and $F_j$, $j\!=\!1,\ldots,n\!-\!1$, generate the Hopf *-algebra
$\slrm$, and \rf[Anfang]--\rf[Ang2] define a
$\slrm$-action on $\qhn$ such that $\qhn$ becomes a
$\slrm$-module *-algebra. This action on $\qhn$ is well-known
because it can be obtained from a $\cO(\mathrm{SL}_q(n,\R))$-coaction
and a dual pairing of $\slrm$ and $\cO(\mathrm{SL}_q(n,\R))$ 
(see Sections 1.3.5, 9.3.3 and 9.3.4 in \cite{KS}). 
Now Proposition \ref{qhn-haar}
asserts that we can develop a $\slrm$-invariant integration theory
on $\qhn$, that is, the *-algebras $\FFD$ and $\BBA$
are $\slrm$-module *-algebras and Equation \rf[BH] defines a
$\slrm$-invariant functional on both algebras. Note furthermore 
that, under the assumptions of Lemma \ref{l-ABrho}, we obtain a
\mbox{*-rep}\-re\-sen\-ta\-tion $\pi:\slrm\rightarrow\ld$ by assigning
$\pi(K_j)=\rho_j$, $\pi(E_j)=A_j$ and $\pi(F_j)=B_j$
for $j=1,\ldots,n-1$.

\section*{Acknowledgments}
%
E.W. thanks Klaus-Detlef K\"ursten for his support and useful comments. 
This work was partially supported by the DFG fellowship WA 1698/2-1.

\end{document}